\documentclass{amsart}

\usepackage[centertags]{amsmath}
\usepackage{amsfonts}
\usepackage{amssymb}
\usepackage{amsthm}
\usepackage[colorlinks, hidelinks]{hyperref}
\usepackage{tikz}
\usepackage{tikz-cd}
\usepackage[normalem]{ulem}
\usepackage[shortlabels]{enumitem} 
\usepackage{marginnote}
\usepackage[all]{xy}
\usepackage{bbm}
\usepackage{mathtools}

\usepackage[normalem]{ulem}

\newcommand{\N}{\mathbb{N}}
\newcommand{\R}{\mathbb{R}}

\newcommand{\Z}{\mathbb{Z}}

\newcommand{\Natural}{\mathbb N}

\newcommand{\Real}{\mathbb R}

\newcommand{\set}[1]{\left\{#1\right\}}

\newcommand\restr[2]{{
		\left.\kern-\nulldelimiterspace 
		#1 
		\littletaller 
		\right|_{#2} 
}}
\newcommand{\littletaller}{\mathchoice{\vphantom{\big|}}{}{}{}}

\definecolor{egraf}{rgb}{0.2,0.4,0}

\newcommand{\diam}{\mathop{\mathrm{diam}}\nolimits}

\newcommand{\Sz}{\mathop{\mathrm{Sz}}\nolimits}

\newcommand{\norm}[1]{\left\Vert#1\right\Vert}
\newcommand{\duality}[1]{\left\langle#1\right\rangle}

\newcommand{\closedball}[1]{B_{#1}}

\newcommand{\Free}{{\mathcal F}}
\newcommand{\Lip}{{\mathrm{Lip}}}

\newcommand{\F}{\mathcal{F}}

\newcommand{\eps}{\varepsilon}

\def\<{\langle}
\def\>{\rangle}

\newcommand{\mol}[1]{m_{#1}}                                

\theoremstyle{plain}
\newtheorem{thm}{Theorem}[section]
\newtheorem{theorem}[thm]{Theorem}

\newtheorem{corollary}[thm]{Corollary}
\newtheorem{lem}[thm]{Lemma}
\newtheorem{lemma}[thm]{Lemma}
\newtheorem{prop}[thm]{Proposition}
\newtheorem{proposition}[thm]{Proposition}

\theoremstyle{definition}

\newtheorem{definition}[thm]{Definition}

\newtheorem{remark}[thm]{Remark}

\author[E. Basset]{Estelle Basset}
\address[E. Basset]{Universit\'e Marie et Louis Pasteur, CNRS, LmB (UMR 6623), F-25000 Besan\c con, France.}
\email{estelle.basset@univ-fcomte.fr}

\author[G. Lancien]{Gilles Lancien}
\address[G. Lancien]{Universit\'e Marie et Louis Pasteur, CNRS, LmB (UMR 6623), F-25000 Besan\c con, France.}
\email{gilles.lancien@univ-fcomte.fr}

\author[A. Proch\'azka]{Anton\'in Proch\'azka}
\address[A. Proch\'azka]{Universit\'e Marie et Louis Pasteur, CNRS, LmB (UMR 6623), F-25000 Besan\c con, France.}
\email{antonin.prochazka@univ-fcomte.fr}

\title{Diversity of Lipschitz-free spaces over countable complete discrete metric spaces}

\begin{document}

\begin{abstract} 
We show that there are uncountably many mutually non-isomorphic Lipschitz-free spaces over countable, complete, discrete metric spaces.
Also there is a countable, complete, discrete metric space whose free space does not embed into the free space of any uniformly discrete metric space.
This enhanced diversity is a consequence of the fact that the dentability index $D$ presents a highly non-binary
behavior when assigned to the free spaces of metric spaces outside of the oppressive confines of compact purely 1-unrectifiable spaces. 
Indeed,
the cardinality of $\{D(\Free(M)): M$ countable, complete, discrete$\}$ is uncountable while $\{D(\Free(M)):M$  infinite, compact, purely 1-unrectifiable$\}=\{\omega,\omega^2\}$.
Similar barrier is observed for uniformly discrete metric spaces as higher values of the dentability index are excluded for their  free spaces:\\ $\{D(\Free(M)):M$ infinite, uniformly discrete$\}=\set{\omega^2,\omega^3}$. 
\end{abstract}

\maketitle

\section{Introduction}
The isomorphic theory of Lipschitz-free spaces (free spaces henceforth), roughly speaking, tries to decide under which conditions on metric spaces $M$ and $N$ the corresponding free spaces $\Free(M)$ and $\Free(N)$ are, resp. are not, isomorphic.
The body of knowledge in both directions is slowly growing. 
Let us mention \cite{DutrieuxFerenczi,Kaufmann,GartlandFreeman,AlbiacAnsorenaCuthDoucha} which provide techniques for producing non-trivial isomorphisms between free spaces. On the other hand the papers  \cite{GodefroyKalton,NaorSchechtmann,HajekLancienPernecka,ANPP,AGPP,BaudierGartlandSchlumprecht,Basset} point out examples of free spaces which are not isomorphic as well as techniques permitting to disprove the existence of isomorphism between given $\Free(M)$ and $\Free(N)$. 
Generally speaking, such techniques are based on studying the isomorphic properties of free spaces. 
The present article falls within this second category and can be considered a follow-up work to~\cite{Basset}.
In~\cite{Basset}, technical results (see below) concerning the weak-fragmentability index $\Phi$ of free spaces lead to the following findings:
\begin{itemize}
    \item there is a  countable uniformly discrete metric space $M$
    such that $\Free(M)$ is not isomorphic to the free space of any compact metric space and does not embed into the free space of any compact purely 1-unrectifiable metric space.
    \item there is a purely 1-unrectifiable metric space $M$ whose free space does not embed into the free space of any uniformly discrete metric space.
    \item there is no separable purely 1-unrectifiable metric space  universal for the class of countable complete metric spaces and bi-Lipschitz embeddings.
\end{itemize}
Here, using technical results (see the abstract and below) concerning the dentability index $D$ of free spaces, we are going to strengthen the above claims as follows:
\begin{itemize}
    \item there is a countable complete discrete metric space $M$ whose free space does not embed into the free space of any uniformly discrete metric space (see Corollary~\ref{c:CCD-cpt}). This answers a question asked by Eva Perneck\'a and Ram\'on Aliaga.
    \item there is no separable purely 1-unrectifiable metric space universal for the class of countable complete discrete spaces and bi-Lipschitz embeddings (see Corollary~\ref{c:NoUniversal}).
\end{itemize}

We now discuss the lower-level results behind the above claims.
Let us recall that it is known (see~\cite{AGPP}) that $\Free(M)$ has the Radon-Nikod\'ym property (RNP) if and only if it has the point of continuity property (PCP) if and only if $M$ is purely-1-unrectifiable (p-1-u). 
In this context two ordinal indices defined for a Banach space $X$ are highly relevant: the dentability index $D(X)$ and the weak fragmentability index $\Phi(X)$. 
Each of these indices gives rise to an uncountable number of isomorphic properties corresponding to the classes on the right-hand side in the following well-known equalities:
\[
\set{X:X \mbox{ is separable with the RNP}}=\bigcup_{\alpha<\omega_1} \set{X:X \mbox{ is separable and }D(X)\leq \alpha}
\]
and
\[
\set{X:X \mbox{ is separable with the PCP}}=\bigcup_{\alpha<\omega_1} \set{X:X \mbox{ is separable and }\Phi(X)\leq \alpha},
\]
where $\omega_1$ is the first uncountable ordinal. The goal of the present paper is to study the behavior of these two indices when applied to the free spaces over metric spaces coming from the following natural classes of p-1-u metric spaces:
\begin{itemize}
    \item compact/proper p-1-u metric spaces,
    \item separable, uniformly discrete metric spaces,
    \item separable, complete, discrete metric spaces.
\end{itemize}

We obtain the following results: the dentability index of $\F(M)$ is at most $\omega^3$ when $M$ is uniformly discrete (see Theorem~\ref{th:D(F(grid))} and Corollary~\ref{c:DentabilityCUD}), and it is at most $\omega^2$ when $M$ is a proper p-1-u metric space (see Remark~\ref{rem:p-1-uProper}).
The estimate for $M$ uniformly discrete is based on the following inequality of independent interest: $D(X)\leq \Phi(L_2(X))$, see \cite[Lemma 2.6]{Lancien1995}.
We provide an example of  a uniformly discrete metric subspace $M$ of $L_1$ such that $D(\F(M))=\omega^3$ (Theorem~\ref{th:D(F(grid))}).
Finally, in Theorem~\ref{thm:mainexample} we show that there exist countable complete discrete metric spaces whose free spaces have arbitrarily high dentability index.

\section{Preliminaries}

\subsection{General notation}

All Banach spaces will be over the field $\R$. Given two Banach spaces $X$ and $Y$, we denote $B(X,Y)$ the space of all linear bounded operators from $X$ to $Y$, and we write $X^* = B(X, \R)$ the topological dual of $X$. We denote the closed unit ball of $X$ by $B_X$, and its unit sphere by $S_X$. 

Given a measure space $(\Omega,\Sigma,\mu)$ and $p \in [1, +\infty)$, we write $L^p(X, \mu)$ the space of Bochner-measurable functions $f \colon \Omega \to X$ such that $|f|^p$ is integrable with respect to $\mu$, equipped with the norm $\norm{f}_p \coloneqq \big(\int_\Omega f^p d\mu\big)^{1/p}$.

Given two metric spaces $(M, d_M)$ and $(N, d_N)$ and $D \in [1, +\infty)$, we say that $M$ \textit{bi-Lipschitz embeds into $N$ with distortion $D$} if there exists $s \in (0, +\infty)$ and a map $f \colon M \to N$ satisfying, for every $x, y \in M$:
\[
s\cdot d_M(x,y) \leq d_N(f(x),f(y)) \leq s\cdot D\cdot d_M(x,y).
\]

If $\set{(M_i, d_i)}_{i \in I}$ is a collection of metric spaces, 
we say that $\set{M_i}_{i \in I}$
\emph{embeds with equal distortion into $N$}
if there exists $D \in [1,+\infty)$
such that for every $i \in I$, there exist a map $f_i \colon M_i \to N$ and a scaling factor $s_i \in (0,+\infty)$ satisfying, for every $x, y \in M_i$:
\[
s_i\cdot d_i(x,y) \leq d_N(f_i(x),f_i(y)) \leq s_i\cdot D\cdot d_i(x,y).
\]
This is similar to but weaker than the situation when $\set{M_i}_{i \in I}$ 
\emph{equi-bi-Lipschitz embeds into $N$}, i.e. when 
there exist $D \in [1, +\infty)$ and $s \in (0, +\infty)$ such that for every $i \in I$, there exists a map $f_i \colon M_i \to N$  satisfying, for every $x, y \in M_i$:
\[
s\cdot d_i(x,y) \leq d_N(f_i(x),f_i(y)) \leq s\cdot D\cdot d_i(x,y).
\]
Note that if $N$ is a Banach space, then $\set{M_i}_{i \in I}$ embeds with equal distortion into $N$ if and only if it equi-bi-Lipschitz embeds into $N$ if and only if there exists $D \in [1, +\infty)$ such that for every $i \in I$, there exists a map $f_i \colon M_i \to N$ satisfying, for every $x, y \in M_i$:
\[
d_i(x,y) \leq \|f_i(x),f_i(y)\|_N \leq D \cdot d_i(x,y).
\]

Given $a, b \in (0,\infty)$, we say that a subset $N$ of a metric space $(M, d)$ is a \textit{$(a, b)$-net of $M$} when the two following conditions are satisfied:
\begin{enumerate}
    \item[-] $N$ is $a$-separated, that is, $d(x, y) \geq a$ for every $x \neq y \in N$;

    \item[-] $N$ is $b$-dense in $M$, that is, for every $x \in M$, there exists $y \in N$ such that $d(x, y) \leq b$.
\end{enumerate}
We say that $N$ is a \textit{net of $M$} whenever there exist $a, b \in (0,\infty)$ such that $N$ is a $(a,b)$-net of $M$.

We denote by $\Z^{<\omega}$ the set of all finitely supported sequences of integers. Let now $X$ be a Banach space equipped with a Schauder basis $\mathcal B=(e_n)_{n =1}^\infty$, we call \textit{integer grid of $\mathcal B$} the following subset of $X$:
\[
\Z_{\mathcal B} \coloneqq \Big\{\sum\limits_{n=1}^\infty a_n e_n,\  (a_n)_{n=1}^\infty \in \Z^{<\omega}\Big\}.
\]
We simply denote $\Z_{c_0}$ the integer grid of $c_0$ equipped with its canonical basis. Let us recall that $\Z_{c_0}$ has the particular property to also be a $(1, 1)$-net of $c_0$. 
We shall also use $\Z_{Haar}$ the integer grid of the normalized Haar basis of $L_1=L_1([0,1])$.

\subsection{Lipschitz-free spaces}

By a \emph{pointed metric space} we mean a metric space $M$ containing a distinguished point, denoted 0 and called \emph{basepoint} or \emph{origin}. If $M$ and $N$ are pointed metric spaces then $\Lip_0(M,N)$ stands for the set of all Lipschitz functions $f :M \to N$ such that $f(0) = 0$. 
If $N=X$ is a real Banach space, then $\Lip_0(M,X)$ is a Banach space when equipped with the Lipschitz norm:
$$ \forall f \in \Lip_0(M,X), \quad \|f\|_L = \sup_{x \neq y \in M} \frac{\|f(x)-f(y)\|_X}{d(x,y)}.$$
In the case $X = \R$, we denote $\Lip_0(M):=\Lip_0(M,\R)$.
Then, for $x\in M$, we let $\delta_M(x) \in \Lip_0(M)^*$ be the evaluation functional defined by $\langle\delta_M(x) , f \rangle  = f(x)$, for all  $f$  in $\Lip_0(M)$. 
It is readily seen that $\delta_M: x \in M \mapsto \delta_M(x) \in \Lip_0(M)^*$ is an isometry. 
The \textit{Lipschitz-free space over $M$} (free space for brevity) is then defined as the closed linear span in $\Lip_0(M)^*$ of all such evaluation functionals:
$$\F(M) := \overline{ \mbox{span}}^{\| \cdot  \|}\left \{ \delta_M(x) \, : \, x \in M  \right \} \subseteq \Lip_0(M)^*.$$
For a classical reference on free spaces, we refer the reader to~\cite{Weaver2}.

Let us now recall the universal extension property of free spaces. Let  $X$ be a Banach space and $f \in \Lip_{0}(M,X)$. Then there exists a unique $\overline{f} \in B(\F(M),X)$ such that $f = \overline{f} \circ \delta_M$. In other words, we have the following commutative diagram:  

$$ \xymatrix{
			M \ar[r]^f \ar[d]_{\delta_M}  & X  \\
			{\F(M)} \ar[ur]_{\overline{f}} } 
$$
      Moreover $\| \overline{f} \|_{B(\F(M),X)} = \| f \|_L$.
Therefore, $\Lip_0(M,X)$ is linearly isometric to $B(\F(M),X)$ and, in particular, $\F(M)^*$ is linearly isometric to $\Lip_0(M)$. Then the weak$^*$ topology induced by $\F(M)$ on $\Lip_0(M)$ coincides on bounded subsets of $\Lip_0(M)$ with the topology of pointwise convergence. 

Another application of the universal extension property is the following factorization result. Let $M,N$ be pointed metric spaces and $f:M\to N$ a Lipschitz map satisfying $f(0)=0$. Then there exists a unique $\widehat f\in B(\F(M),\F(N))$ such that $\widehat f \circ \delta_M=\delta_N \circ f$. Moreover $\| \overline{f} \|_{B(\F(M),\F(N))} = \Lip(f)$. Thus, the following diagram commutes:

$$\xymatrix{
	M \ar[r]^f \ar[d]_{\delta_M}  & N \ar[d]^{\delta_N} \\
	\F(M) \ar[r]_{\widehat{f}} & \F(N).
}
$$
It follows easily from this factorization that if a metric space $M$ bi-Lipschitz embeds into a metric space $N$, then $\F(M)$ is linearly isomorphic to a subspace of $\F(N)$.

A fundamental tool will be McShane-Whitney's extension theorem, which allows to extend a real valued Lipschitz map defined on a subset $K$ of $M$ to the whole space $M$ without increasing its Lipschitz constant; see e.g. \cite[Theorem 1.33]{Weaver2}. As a first application one gets that  if $0 \in K \subset M$, then $\F(K)$ is linearly isometric to $\F_M(K)$ where: 
$$\F_M(K) := \overline{\text{span}} \{ \delta_M(x) \; | \; x \in K\} \subseteq \F(M).$$
This canonical identification will be used throughout this paper. 

Next, for $x \neq y  \in M$, the \emph{molecule} $m_{x, y} \in S_{\F(M)}$ is defined by:
$$ m_{x, y} := \frac{\delta_M(x) - \delta_M(y)}{d(x,y)}.$$

We conclude this introduction on Lipschitz-free spaces by recalling a fundamental decomposition result due to N. Kalton (Proposition 4.3 in \cite{Kalton2004}).

\begin{proposition}[Kalton's decomposition]
    \label{Kalton's decomposition}
    Let $(M, d)$ be a pointed metric space. For $k \in \Z$, we set $M_k = \set{x \in M,\, d(x, 0) \leq 2^k}$. 
    Then $\F(M)$ is linearly isomorphic to a subspace of $\big(\sum_{k \in \Z} \mathcal{F}(M_k)\big)_{\ell_1}$.  
\end{proposition}

\subsection{Various fragmentability indices of Banach spaces}

Let us describe a general peeling scheme which is frequently used to assign some isomorphically invariant ordinal index to a given Banach space $X$. Assume that $\mathcal A$ is a collection of open subsets of $X$ that is stable under dilations and translations and let $C$ be a subset of $X$. 
For $\varepsilon>0$ we define the set derivation
\[
[C]_\varepsilon'=C \setminus \bigcup \set{A \in {\mathcal A}: \diam(A \cap C)<\varepsilon}
\]
and we put
\[
[C]_\varepsilon^0:=C, \quad [C]_\varepsilon^{\alpha+1}:=[[C]_\varepsilon^\alpha]_\varepsilon' \quad\mbox{ and }\quad [C]_\varepsilon^{\beta}:=\bigcap_{\alpha<\beta} [C]_\varepsilon^\alpha
\]
for every ordinal $\alpha$ and every limit ordinal $\beta$.
Further we define 
\[
\iota_{\mathcal A}(X,\varepsilon):=\inf \set{\alpha: [\closedball{X}]_\varepsilon^\alpha=\emptyset} \mbox{ and } \iota_{\mathcal A}(X):=\sup_{\varepsilon>0} \iota_{\mathcal A}(X,\varepsilon),
\]
adopting the convention that $\inf \emptyset = \infty$ and $\alpha < \infty$ for every ordinal $\alpha$. It is clear that $\iota_{\mathcal A}(X)<\infty$ if and only if, for every bounded subset $C$ of $X$ and every $\eps>0$, there exists $A \in \mathcal A$ such that $A \cap C \neq \emptyset$ and $\diam (A \cap C)<\eps$. Then, we say that $X$ is \emph{fragmentable by $\mathcal A$}. Then the value of $\iota_{\mathcal A}(X)$  measures how fast or efficiently this fragmentability operates and we call it the \emph{$\mathcal A$-fragmentability index of $X$}.
Naturally, if $\mathcal A$ is a class stable under linear isomorphisms, then $\iota_{\mathcal A}(X)$ is an isomorphic invariant. 

\medskip
We now detail the three main examples that will be considered in this paper. 
\begin{enumerate}
    \item If we consider $\mathcal A$ to be the collection of weak open subsets of a Banach space $X$, the index $\iota_{\mathcal A}(X)$ is called the \emph{weak fragmentability index of $X$} and for $C$ a subset of $X$, we denote $\sigma_\varepsilon^\alpha(C):=[C]_\varepsilon^\alpha$. Then $\Phi(X,\eps):=\iota_{\mathcal A}(X,\eps)$ and $\Phi(X):=\iota_{\mathcal A}(X)$. The index $\Phi(X)$ is clearly an isomorphic invariant. 
    \item We call an \emph{open slice} of a Banach space $X$ a set of the form $S=\{x\in X,\ x^*(x)>a\}$, with $x^*\in X^*$ and $a\in \R$. If $\mathcal A$ is the collection of open slices of $X$, then $\iota_{\mathcal A}(X)$ is called the \emph{dentability index} of $X$ and for $C$ a subset of $X$, we denote $d_\varepsilon^\alpha(C):=[C]_\varepsilon^\alpha$. Then $D(X,\eps):=\iota_{\mathcal A}(X,\eps)$ and $D(X):=\iota_{\mathcal A}(X)$. Again, the index $D(X)$ is  an isomorphic invariant and, since open slices are weakly open, $\Phi(X)\le D(X)$. 
    \item If $X=Y^*$ is a dual Banach space and $\mathcal A$ is the collection of weak$^*$ open sets of $Y^*$, the index $\iota_{\mathcal A}(X)$ is called the \emph{Szlenk index of $Y$} and for $C$ subset of $Y^*$, we denote $s_\varepsilon^\alpha(C):=[C]_\varepsilon^\alpha$. Then $Sz(Y,\eps):=\iota_{\mathcal A}(X,\eps)$ and $Sz(Y):=\iota_{\mathcal A}(X)$. It is important to note that $Sz(Y)$ is an isomorphic invariant for $Y$, but not for $X$, which explains the notation. Since weak$^*$ open sets in $Y^*$ are weakly open, we have that $\Phi(X)\le Sz(Y)$. 
    \item If $X=Y^*$ is a dual Banach space and $\mathcal A$ is the collection of weak$^*$ open slices of $Y^*$, the index $\iota_{\mathcal A}(X)$ is called the \emph{weak$^*$ dentability index of $Y$} and denoted $Dz(Y)$. Clearly $D(X^*)\le Dz(Y)$ and $Sz(Y)\le Dz(Y)$. 
\end{enumerate}

\medskip The following elementary proposition is well known for the dentability or Szlenk indices. We state  and quickly prove it in the general setting described above.

\begin{prop}\label{prop:conv_comb} Let $X$ be a Banach space. Assume that $\mathcal A$ is a collection of open subsets of $X$ that is stable under dilations and translations and denote $[C]'_\eps$ the derivation associated with $\mathcal A$.  Then for any $\eps \in (0,1]$, any $\lambda \in (0,1)$ and any ordinal $\alpha$:
$$\lambda [B_X]^\alpha_\eps + (1-\lambda)B_X\subset [B_X]^\alpha_{\lambda\eps}.$$
\end{prop}

\begin{proof} The proof is an easy transfinite induction. Clearly, the statement is true for $\alpha=0$ and passes to limit ordinal. So assume it is true for $\alpha$ and let us prove it for $\alpha+1$. Let $z=\lambda x+(1-\lambda)y$, with $x\in [B_X]^{\alpha+1}_\eps$ and $y\in B_X$. By the induction hypothesis, $z\in [B_X]^\alpha_{\lambda\eps}$ and we need to show that $z\in [B_X]^{\alpha+1}_{\lambda\eps}$. So, let $V \in \mathcal A$ such that $z\in V$. Then $W=\frac{1}{\lambda}V-\frac{1-\lambda}{\lambda}y \in \mathcal A$ and $x\in W$. Since $x \in [B_X]^{\alpha+1}_\eps$, $\diam(W \cap [B_X]^{\alpha}_\eps)\ge \eps$. Since 
$\lambda (W \cap [B_X]^{\alpha}_\eps)+(1-\lambda)y \subset  V\cap [B_X]^\alpha_{\lambda\eps}$, we deduce that $\diam (V\cap [B_X]^\alpha_{\lambda\eps}) \ge \lambda \eps$, which shows that $z\in [B_X]^{\alpha+1}_{\lambda\eps}$ and concludes the proof.   
\end{proof}

We also recall a result on the weak fragmentability index of $\ell_1$-sums. Let $(X_n)_{n=1}^\infty$ be a sequence of Banach spaces. We denote 
$$\Big(\sum_{n \in \N} X_n\Big)_{\ell_1}=\Big\{(x_n)_{n=1}^\infty \in \prod_{n=1}^\infty X_n,\ \sum_{n=1}^\infty \|x_n\|_{X_n} <\infty\Big\},$$
equipped with the norm $\|(x_n)_{n=1}^\infty\|=\sum_{n=1}^\infty \|x_n\|_{X_n}$. We have the following (see Proposition 3.9 in \cite{Basset}).

\begin{prop}
    \label{prop: fragmentabilty index of l_1 sums}
Let $(X_n)_{n=1}^\infty$ be a sequence of Banach spaces and assume that there exists an ordinal $\alpha$ such that $\Phi(X_n)\le \alpha$ for all $n\in \N$. Then 
$$\Phi\Big(\Big(\sum_{n \in \N} X_n\Big)_{\ell_1}\Big)\le \alpha\cdot\omega,$$
where $\omega$ is the first infinite ordinal.
\end{prop}

We conclude with a very elementary lemma that we will use repeatedly.

\begin{lemma}\label{lem:midpointtrick}
Let $C$ be a closed convex set in a Banach space $X$ and $\eps>0$. Assume that $x,y\in C$ are such that $\|x-y\|\ge 2\eps$. Then $\frac{x+y}{2} \in d_\eps(C)$.
\end{lemma}
\begin{proof} Let $S$ be an open slice of $X$ containing $\frac{x+y}{2}$. Then $S$ contains either $x$ or $y$, so its diameter is at least $\eps$. It follows that $\frac{x+y}{2} \in d_\eps(C)$.
\end{proof}

\subsection{The Radon-Nikod\'{y}m and the Point of Continuity properties}
We refer the reader to the book by J. Diestel and J.J. Uhl \cite{DiestelUhl} for all the necessary background and historical references on the Radon-Nikod\'{y}m Property. Let us just recall that a Banach space $X$ has the \emph{Radon-Nikod\'{y}m Property} (RNP in short) if for every finite measure space $(\Omega,\Sigma,\mu)$ and every vector measure $F:\Sigma \to X$ that is absolutely continuous with respect to $\mu$, there exists $f\in L_1(\mu,X)$ such that $F(A)=\int_A f\,d\mu$ for all $A\in \Sigma$. In this paper we shall focus on the geometric characterization of this property in terms of the dentability of bounded subsets of $X$. A bounded subset $C$ of a Banach space $X$ is said to be \emph{dentable} if for every $\eps>0$, there exists an open slice $S$ in $X$ such that $S\cap C \neq \emptyset$ and $\diam (S \cap C)<\eps$. Then a Banach space $X$ has the Radon-Nikod\'{y}m Property if and only if every non empty bounded subset of $X$ is dentable (see Theorem V.7 in \cite{DiestelUhl} for instance). Using this characterization, the following is then easy to see.

\begin{prop}\label{prop:RNP_D(X)} Let $X$ be a separable Banach space. Then $X$ has the RNP if and only if $D(X)<\omega_1$, where $\omega_1$ is the first uncountable ordinal. 
\end{prop}

We now recall that a Banach space $X$ has the \emph{Point of Continuity Property} (PCP in short) if for every non empty bounded closed subset $C$ of $X$ the identity map from $(C,w)$ to $(C,\|\ \|)$ is continuous at some point $x\in C$ (where $w$ stands for the weak topology on $X$). A Banach space $X$ has PCP  if and only if every non empty bounded closed subset $C$ of $X$ is weakly fragmentable (see \cite{Phelps} for instance), meaning that for every $\eps>0$, there exists a weakly open subset $V$ of $X$ such that $V\cap C \neq \emptyset$ and $\diam (V \cap C)<\eps$. The following is then classical. We refer to \cite[Proposition 2.5]{Basset} for a detailed proof. 

\begin{prop} Let $X$ be a separable Banach space. Then $X$ has the PCP if and only if $\Phi(X)<\omega_1$. 
\end{prop}

We conclude this section by recalling that a dual Banach space $X^*$ has RNP if and only if $X$ is an Asplund space if and only if every separable subspace of $X$ has a separable dual (see  \cite[Section III.3]{DiestelUhl} or \cite{DevilleGodefroyZizler}). 
The following will be relevant for this work (see Theorem 1 in \cite{Lancien2006}). 

\begin{prop} Let $X$ be a separable Banach space. Then the following are equivalent
\begin{enumerate}[(i)]
    \item $X^*$ has the RNP,
    \item $X^*$ is separable,
    \item $Sz(X)<\omega_1$. 
\end{enumerate}
    
\end{prop}

\subsection{The countably branching diamond graphs}\label{sec:diamonds}

We define $D_1^\omega=\set{t_1,b_1} \cup \set{x_i:0\leq i< \omega}$. 
We consider the complete bipartite graph on these vertices (edges connect every vertex in the first set to every vertex in the second set). The edge $\set{t_1,x_i}$ will be denoted $(i,+)$ and the edge $\set{b_1,x_i}$ will be denoted $(i,-)$. We define $d_1 = d_1^\omega$ on $D_1^\omega$ as the shortest path metric corresponding to this graph structure, i.e. $d_1(t_1,b_1)=2$.

When $D_n^\omega$ has been defined we define $D_{n+1}^\omega$ as the graph (together with its shortest path metric $d_{n+1}$) where each edge in $D_1^\omega$ has been replaced by $D_n^\omega$. This copy is called $D_n^{(j,\pm)}$ where $(j,\pm)$ corresponds to  the edge being replaced.
The ``$2+\omega$'' vertices of this original $D_1^\omega$ are denoted $t_{n+1}$, $b_{n+1}$, $x^i_{n+1}$ (see Figure~\ref{fig1}). Thus $d_n(t_n, b_n) = 2^n$.

\medskip
The spaces $D_n^\omega$, $n\in \Natural$, are countable uniformly discrete metric spaces. The family $(D_n^\omega)_{n \in \N}$ is called the family of \emph{countably branching diamond graphs}. For further details on these graphs, we refer the reader to \cite{BCDKRSZ}. 

\medskip
Diamond graphs play an important role in Banach space theory. For instance, the bi-Lipschitz embeddability of the binary diamond graphs $(D_n^2)_{n \in \N}$ (where the first diamond $D_1^2$ is defined with two middle points instead of ``$\omega$'' middle points as for $D_1^\omega$) characterizes the non super-reflexivity (see \cite{JohnsonSchechtman}).

\newpage

\begin{figure}[h]
\caption{$D_1$ and $D_n$, $n \geq 2$}\label{fig1}
\centering{
\includegraphics[trim={25mm 175mm 7cm 25mm},clip]{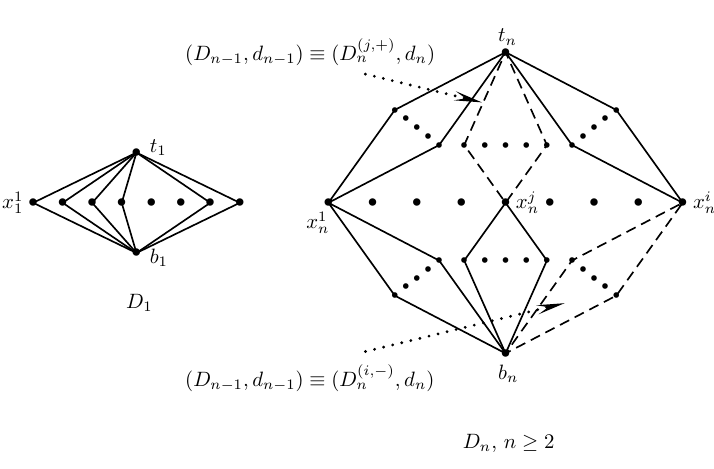}}
\end{figure}

\section{Uniformly discrete metric spaces}\label{sec:uniformly discrete}

In this section we study in details the weak-fragmentability and the dentability index of uniformly discrete metric spaces. 
\subsection{Weak fragmentability index}

For the convenience of the reader, we start by proving in full detail the following folklore upper bound for the weak fragmentability index of a uniformly discrete metric space, that was already mentioned in \cite[Remark 3.12]{Basset}. 
\begin{prop}\label{c:FragmentabilityCUD}
For every uniformly discrete metric space $M$, $\Phi(\F(M))\leq \omega^2$. 
If moreover $M$ is infinite, then $\Phi(\Free(M))$ is either equal to $\omega$ or to $\omega^2$. 
\end{prop}  

\begin{proof} The part concerning infinite metric spaces follows from the first estimate and from the the following well known facts:
\begin{itemize}
    \item for a Banach space $X$, $\Phi(X)=\infty$ or $\Phi(X)=\omega^\alpha$, for some ordinal $\alpha$ (see \cite{Basset}),
    \item for a Banach space $X$, $\Phi(X)\ge \omega$ if and only if $X$ is infinite dimensional. 
\end{itemize} 
It remains to prove that $\Phi(\F(M))\leq \omega^2$, when $M$ is uniformly discrete. 

So, let $M$ be a uniformly discrete metric space. Note that a ball of $M$ is a bounded uniformly discrete metric space. Thus its free space is isomorphic to $\ell_1(\Gamma)$, for some index set $\Gamma$. It is well known that $\Phi(\ell_1(\Gamma))\le \omega$ (actually it is equal to $\omega$ as soon as $\Gamma$ is infinite). Then the conclusion follows from Kalton's decomposition (Proposition \ref{Kalton's decomposition}) and Proposition \ref{prop: fragmentabilty index of l_1 sums}.
\end{proof}

In order to see that our upper bound is attained, the key ingredient will be to use the countably branching diamond graphs $(D_n^\omega)_{n\in \N}$. The following proposition is a direct consequence of the results from \cite{Basset}.

\begin{prop}\label{prop: Phi_free_diamonds} Let $(M,d)$ be a metric space and assume that the family $(D_n^\omega)_{n\in \N}$ embeds with equal distortion into $M$. Then $\Phi(\F(M))\ge \omega^2$.
\end{prop}
\begin{proof} So assume that there exists $C\ge 1$ such that for each $n\in \N$, $D_n^\omega$ bi-Lipschitz embeds into  $M$ with distortion at most $C$. Then $\F(D_n^\omega)$ linearly embeds into $\F(M)$ with distortion at most $C$. Proposition 3.5 in \cite{Basset} ensures that $\sigma_1^n(B_{\F(D_n^\omega)})\neq \emptyset$. It readily follows that $\sigma_{1/C}^n(B_{\F(M)}) \neq \emptyset$ and so, using~\cite[Fact~2.3]{Basset},   $0\in\sigma_{1/2C}^n(B_{\F(M)})$ for every $n\in \N$. 
This, together with~\cite[Fact~2.4]{Basset} implies 
that $\Phi(\F(M))\ge \omega^2$.
\end{proof}

We can now present a few examples of uniformly discrete metric spaces $M$ such that $\Phi(\F(M))=\omega^2$. Let us quickly recall the definition of the $\ell_1$-sum of pointed metric spaces. Let $(M_n)_{n\in \N}$ be a sequence of pointed metric spaces and $M=\{0_M\}\bigsqcup_{n\in \N}(M_n \setminus\{0_{M_n}\})$. In $M$, all base points $0_{M_n}$ are identified with $0_M$. For $x,y \in M_n$, $d_M(x,y)=d_{M_n(x,y)}$ and for $x\in M_n$, $y\in M_m$ and $n\neq m$, $d_M(x,y)=d_{M_n}(x,0_{M_n})+d(y,0_{M_m})$.  It is clear that if all $M_n$'s are $a$-separated for some $a>0$, then $M$ is $a$-separated. Our first example is then immediate.

\begin{proposition} Let $M$ be the $\ell_1$-sum of the $D_n^\omega$'s. Then $M$ is a uniformly discrete metric space and the weak-fragmentability index of $\F(M)$ is equal to $\omega^2$. 
\end{proposition}

We now concentrate on grids and nets in Banach spaces and we start with a study of  $\Z_{c_0}$, which is a universal countable uniformly discrete metric space. Let us recall that  $\Z_{c_0}=(\Z^{<\omega},\|\ \|_\infty)$ is a net in $c_0$ and detail, for the sake of completeness, this rather well known universality result.

\begin{prop}
    \label{prop: grid over c0 universal space}
    $\Z_{c_0}$ is universal for countable uniformly discrete metric spaces and bi-Lipschitz embeddings.
    More precisely, given a countable uniformly discrete metric space $(M,d)$ and $\varepsilon > 0$, there exist $f \colon M \to \Z_{c_0}$ and $s \in (0,+\infty)$ such that for every $x, y \in M$:
    \[
    s\cdot d(x,y) \leq \norm{f(x) - f(y)}_{\infty} \leq s\cdot (2+\varepsilon)d(x,y).
    \]   
\end{prop}

\begin{proof} Let $\delta >0$ such that $(M,d)$ is $\delta$-separated and fix $\eps >0$ and $\lambda>\frac{1}{\delta}$ such that 
$(2\lambda+\frac1\delta)(\lambda -\frac1\delta)^{-1}\leq 2+\varepsilon$. 
Since $(M,d)$ is separable, it follows from the optimal version of Aharoni's theorem proved in \cite{KaltonLancien} that there exists $h:M \to c_0$ such that 
$$\forall x,y \in M,\ \ \lambda d(x,y)\le \|h(x)-h(y)\|_\infty \le 2\lambda d(x,y).$$
Then, define $g:c_0\to \Z_{c_0}$ by approximating each term of a sequence $z \in c_0$ by the closest integer. It is clear that for all $z\in c_0$, $\|z-g(z)\|_\infty\le \frac{1}{2}$, so 
\[\forall z,z' \in c_0,\ \ \norm{z - z'}_{\infty}-1 \leq \norm{g(z) - g(z')}_{\infty} \leq \norm{z - z'}_{\infty}+1\]
Define now $f=g \circ h$. 
Since $(M,d)$ id $\delta$-separated, we get
\[\forall x,y \in M,\ \ \Big(\lambda-\frac{1}{\delta}\Big)d(x,y) \le \|f(x)-f(y)\|_\infty \le \Big(2\lambda+\frac{1}{\delta}\Big)d(x,y).\]
We let $s=\lambda-\frac1\delta$. 
By our initial choice of $\lambda$, it is obvious that $f$ is the embedding we were looking for.
\end{proof}

\begin{proposition}
    \label{prop: Phi(F(grid))}
    The weak-fragmentability index of $\F(\Z_{c_0})$ is equal to $\omega^2$.
\end{proposition}

\begin{proof} Proposition \ref{prop: grid over c0 universal space} ensures that the family $(D_n^\omega)_{n\in \N}$ embeds with equal distortion into $\Z_{c_0}$. Then we apply Proposition \ref{prop: Phi_free_diamonds} to obtain that $\Phi(\F(\Z_{c_0}))\ge\omega^2$. The other inequality follows from Proposition \ref{c:FragmentabilityCUD}. 
\end{proof}

It is proved in \cite{BCDKRSZ} that the family $(D_n^\omega)_{n\in \N}$ equi-bi-Lipschitz embeds into $L_1$. A reasoning similar to the proof of Proposition  \ref{prop: grid over c0 universal space}  yields that $(D_n^\omega)_{n\in \N}$  embeds with equal distortion into any net in $L_1$, and, by looking carefully at the proof in \cite{BCDKRSZ}, that it embeds with equal distortion into the integer grid of the normalized Haar basis of $L_1$ that we denote $\Z_{Haar}$. So, we can deduce the following from Proposition \ref{prop: Phi_free_diamonds}.

\begin{proposition} Let $N$ be a net in $L_1$. Then
$$\Phi(\F(N))=\Phi(\F(\Z_{Haar}))=\omega^2.$$
\end{proposition}

\begin{remark}\label{rem:Phi} 
Determining what are the possible values of $\Phi(\F(M))$ for $M$ separable complete and purely 1-unrectifiable and computing $\Phi(\F(M))$ is a general interesting question. The results in \cite{Basset} imply that this set of possible values is uncountable. One important subquestion is to determine, for which $M$ we have that $\Phi(\F(M))=\omega$. Let us mention in this direction that it follows from results in \cite{AGPP} and \cite{Dalet} that if $M$ is a proper (i.e. its closed balls are compact) p-1-u metric space, then $\F(M)$ is isometrically a dual space and its norm is asymptotically uniformly convex for the corresponding weak$^*$ topology. In particular $\F(M)$ is asymptotically uniformly convex and thus $\Phi(\F(M))\le\omega$.

We do not know any example of a net in an infinite dimensional Banach space or of an integer grid of an infinite semi-normalized Schauder basis such that the  fragmentability index of its free space is equal to $\omega$.  In view of the results from \cite{LancienPernecka} and \cite{HajekPernecka}, the integer grid of the canonical basis of $\ell_1$ is a natural candidate. 
\end{remark}

\begin{remark}\label{r:Johnson-Zippin}
    In~\cite[Example 5.6]{Kalton2004}, Kalton proved that $\Free(\Z_{c_0})$ is not isomorphic to a subspace of the Johnson-Zippin space $C_1$.
    His argument is based on the fact that already the grid $\Z_{c_0}$ does not bi-Lipschitz embed into any stable space, in particular neither into $C_1$.
    Recall that $C_1=\left(\sum_{n=1}^\infty X_n\right)_{\ell_1}$ where $(X_n)_{n=1}^\infty$ is a sequence of finite dimensional spaces dense, for the Banach-Mazur distance, in all finite dimensional spaces. 
    (The space $C_1$ is known to be independent of the particular sequence $(X_n)$ up to almost isometry.)
    It is clear that $\Phi(C_1)=\omega$ by a direct application of Proposition~\ref{prop: fragmentabilty index of l_1 sums}.
    Since $\Phi(\Free(\Z_{c_0}))=\omega^2$ we obtain a fundamentally different proof of Kalton's example. 
    More importantly, if $M$ is any net in $L_1$ or the space $\Z_{Haar}$, then again $\Free(M)$ is not isomorphic to a subspace of $C_1$ by comparison of the corresponding weak-fragmentability indices. 
    This time, on the other hand, $M\subset L_1$ which is known to be a stable space. 
\end{remark}

\subsection{Dentability index}
We now study the dentability index of free spaces over uniformly discrete metric spaces and start again with a general upper bound. 
\begin{prop}\label{c:DentabilityCUD}
For every uniformly discrete metric space $M$, $D(\F(M))\leq \omega^3$. 
If moreover $M$ is infinite then $D(\Free(M))$ is either equal to $\omega^2$ or to $\omega^3$. 
\end{prop}  

\begin{proof} 
The part concerning infinite metric spaces follows from the first estimate and from the the following well known facts:
\begin{itemize}
    \item for a Banach space $X$, $D(X)=\infty$ or $D(X)=\omega^\alpha$, for some ordinal $\alpha$ (see \cite{Lancien2006} for instance),
    \item a non-trivial Banach space $X$ is super-reflexive if and only if $D(X)=\omega$ ($D(X)=1$ only in trivial case $X=\{0\}$).
\end{itemize} 
Indeed, if $M$ is infinite, then it is well known and easily seen that $\Free(M)$ contains a copy of $\ell_1$.
Thus $\omega^2=D(\ell_1)\leq D(\Free(M))\leq \omega^3$.
(Here we have also used that $\omega^2\leq D(\ell_1)\leq Dz(c_0)=\omega^2$).

It remains to prove that $D(\Free(M))\leq \omega^3$ when $M$ is uniformly discrete, which will be a consequence of a number of intermediate results of independent interest.
\end{proof}

To obtain this upper bound we draw inspiration from the proof of Theorem 2 in \cite{HLP}. Given a Banach space $X$, let $L_2(X)$ be the space of all (equivalence classes) of Bochner measurable functions $f:[0,1] \to X$ with respect to the Lebesgue measure $\lambda$ on $[0,1]$ such that $\int_{[0,1]} \|f\|^2_X\,d\lambda <\infty$, equipped with the norm 
$$\|f\|_{L_2(X)}=\Big(\int_{[0,1]} \|f\|^2_X\,d\lambda\Big )^{1/2}.$$  

We will use the following lemma which has been proved in~\cite[Lemma 2.6]{Lancien1995}.

\begin{lemma}\label{prop: link D(X) Phi(L_2(X))}
    Let $X$ be any Banach space. Then $D(X) \leq \Phi(L_2(X))$. 
\end{lemma}

In order to bound above $\Phi(L_2(\Free(D(M))))$ we need the following lemma,  which is an adaptation of a known technical trick (see for instance Lemma 6 in \cite{HLP}).

\begin{lemma}
\label{l:PhiVsDForEllOneSums}
Let $X_n$, $n \in \N$, be Banach spaces. 
Let $Z=(\sum_{k\in \N} X_k)_{\ell_1}$ and for any $N\in \N$ let $Z_N=(\sum_{k=1}^N X_k)_{\ell_1}$ with $P_N:Z \to Z_N$, the canonical projection. 
Let $\Pi_N:L_2(Z) \to L_2(Z_N)$ be defined by $\Pi_N(f)(t)=P_n(f(t))$ for every $f \in L_2(Z)$ and $t \in [0,1]$.

Let $\varepsilon \in (0, 1)$, $f \in B_{L_2(Z)}$ and $N \in \N$ be such that $\norm{\Pi_N f}^2 > 1 - \varepsilon^2$. 
Then, for every ordinal $\alpha$,
\[
    f \in \sigma_{3\varepsilon}^{\alpha}({B_{L_2(Z)}})\ \Longrightarrow\ \Pi_N f \in \sigma_{\varepsilon}^{\alpha}({B_{L_2(Z_N)}}).
\]
\end{lemma}

\begin{proof}
Let us start by observing that it is routine to check that $\Pi_Nf\in L_2(Z_N)$ and so, 
clearly, $\Pi_N$ is a norm one projection from $L_2(Z)$ onto $L_2(Z_N)$. 
Moreover, since $\norm{z}=\norm{P_{N}z}+\norm{z-P_N z}$ for every $z\in Z$, we have
\begin{equation}\label{eq:lower2}
\forall f \in L_2(Z),\ \ \|f\|_{L_2(Z)}^2 \ge   \|\Pi_Nf\|_{L_2(Z)}^2 + \|f-\Pi_Nf\|_{L_2(Z)}^2.  
\end{equation}

We will prove the lemma using a transfinite induction on $\alpha$. 
Since $\norm{\Pi_N}=1$, the statement is clearly true for $\alpha = 0$. 
Assume it is true for every $\beta < \alpha$.  
If $\alpha$ is a limit ordinal, it is also true for $\alpha$ by taking the intersection on $\beta < \alpha$. 
So assume now that $\alpha = \beta + 1$, and let $f \in B_{L_2(Z)}$ and $N \in \N$ such that $\norm{\Pi_N f}^2 > 1 - \varepsilon^2$. 
We proceed by contraposition: assume that $\Pi_N f \notin \sigma_{\varepsilon}^{\alpha}({B_{L_2(Z_N)}})$, and let us show that $f \notin \sigma_{3\varepsilon}^{\alpha}({B_{L_2(Z)}}) = \sigma_{3\varepsilon}\left(\sigma_{3\varepsilon}^{\beta}(B_{L_2(Z)})\right)$. 
If $f \notin \sigma_{3\varepsilon}^{\beta}({B_{L_2(Z)}})$, we are done, so we may assume that $f \in \sigma_{3\varepsilon}^{\beta}({B_{L_2(Z)}})$. 
Then the induction hypothesis implies that $\Pi_N f \in \sigma_{\varepsilon}^{\beta}({B_{L_2(Z_N)}})$, and since $\Pi_N f \notin \sigma_{\varepsilon}^{\beta+1}({B_{L_2(Z_N)}})$, there exists $V$, a weakly open subset of $L_2(Z_N)$ containing $\Pi_N f$ such that $\diam\left(V \cap \sigma_{\varepsilon}^{\beta}({B_{L_2(Z_N)}})\right) < \varepsilon$. 
We may assume that $V$ is of the form 
    \[
    V = \{g \in L_2(Z_N): \forall i \in \{1, \dots, r\},\, \duality{\phi_i,g} > \alpha_i\}
    \]
with $r\in \N$, $\alpha_i \in \R$ and $\phi_i \in L_2(Z_N)^*$. 
   Since $\norm{\Pi_N f} > \sqrt{1 - \varepsilon^2}$, we may also assume, adding one more functional if necessary, that $\alpha_1 > \sqrt{1 - \varepsilon^2}$. 
   This last assumption implies that $V \cap (\sqrt{1 - \varepsilon^2}) B_{L_2(Z_N)} = \emptyset$.
   We define the following weakly open subset of $L_2(Z)$: 
    \[U = \Big\{g \in L_2(Z): \forall i \in \{1, \dots, r\},\, \duality{\Pi_N^*\phi_i,g} > \alpha_i\Big\}=\Pi_N^{-1}(V).\] 
    We have that $g \in U$ if and only if $\Pi_N g \in V$. 
    In particular, $U$ is a weakly open set containing $f$, so $f \in U \cap \sigma_{3\varepsilon}^{\beta}({B_{L_2(Z)}})$ and to conclude that $f \notin \sigma_{3\varepsilon}^{\beta+1}({B_{L_2(Z)}})$, it will be enough to show that $\diam\left(U \cap \sigma_{3\varepsilon}^{\beta}({B_{L_2(Z)}})\right) < 3 \varepsilon$.

    So let us consider $g, g' \in U \cap \sigma_{3\varepsilon}^{\beta}({B_{L_2(Z)}})$. 
    Then $\Pi_N g, \Pi_N g'$ belong to $V$ and thus are of norm greater than $\sqrt{1 - \varepsilon^2}$. 
    Since $g, g'$ are in the unit ball, it follows from \eqref{eq:lower2} that $\norm{g - \Pi_N g} < \varepsilon$ and $\norm{g' - \Pi_N g'} < \varepsilon$. Moreover, the induction hypothesis implies that 
    \[
    \begin{split}
        \norm{g - g'} &\leq \norm{g - \Pi_N g} + \norm{\Pi_N g - \Pi_N g'} + \norm{\Pi_N g' - g'} \\
        &\leq 2 \varepsilon + \diam\left(V \cap \sigma_{\varepsilon}^{\beta}({B_{L_2(Z_N)}})\right) < 3\varepsilon.
    \end{split}
    \]
    Therefore, $\diam\left(U \cap \sigma_{3\varepsilon}^{\beta}({B_{L_2(Z)}})\right) < 3\varepsilon$. This finishes the proof of this lemma.
\end{proof}

\begin{proof}[Proof of the fact that $D(\Free(M))\leq \omega^3$ when $M$ is uniformly discrete]

Let $M$ be uniformly discrete.
Then Kalton's decomposition (Proposition \ref{Kalton's decomposition}) implies that $\Free(M)$ is isomorphic to a subspace of $Z=(\sum_{k \in \N} \F(M_k))_{\ell_1}$.
We will show that $\Phi(L_2(Z))\leq \omega^3$ which, by Lemma \ref{prop: link D(X) Phi(L_2(X))}, will finish the proof. 

First, for $N\in \N$ denote $Z_N=(\sum_{k =1}^N \F(M_k))_{\ell_1}$. Since $M_k$ is  uniformly discrete and bounded for each $k\leq N$, we have that $Z_N$ is isomorphic to $\ell_1(\Gamma)$, for some index set $\Gamma$. Then $\phi(L_2(\ell_1(\Gamma)))\le Sz(L_2(c_0(\Gamma)))$ (recall that the dual of $L_2(c_0(\Gamma))$ is isometric to $L_2(\ell_1)$, because $c_0(\Gamma)^*$ is isometric to $\ell_1(\Gamma)$ and thus has the RNP). 
Now, since the Bochner measurable functions are essentially separably valued, it is easy to check that any separable subspace $X$ of $L_2(c_0(\Gamma))$ is isometric to a subspace of $L_2(c_0(\Delta))$, for some countable subset $\Delta$ of $\Gamma$.  It has been proved in \cite{HLP} that $\Sz(L_2(c_0)) = \omega^2$. Thus $Sz(X)\le \omega^2$. Then it follows from the separable determination of the Szlenk index \cite{Lancien1996} that $Sz(L_2(c_0(\Gamma)))\le \omega^2$. 
Put all together we have $\Phi(L_2(Z_N))\leq \omega^2$ for every $N \in \N$.

 
Now, let $\varepsilon \in (0, 1)$ and assume that $f \in B_{L_2(Z)}$ is such that $\norm{f} > \sqrt{1 - \varepsilon^2}$. 
By Lebesgue's dominated convergence theorem, there exists $N \in \N$ such that $\norm{\Pi_N f} > \sqrt{1 - \varepsilon^2}$. 
Since $\Phi(L_2(Z_n))\le \omega^2$, we get that $\sigma_{\varepsilon/3}^{\omega^2}({B_{L_2(Z_N)}}) = \emptyset$. 
So, by Lemma~\ref{l:PhiVsDForEllOneSums}, $f \notin \sigma_{\varepsilon}^{\omega^2}({B_{L_2(Z)}})$. 
Hence, $\sigma_{\varepsilon}^{\omega^2}({B_{L_2(Z)}})\subset (\sqrt{1 - \varepsilon^2}) B_{L_2(Z)}$, and 
an analogue of Lemma 3 in \cite{HLP} (see also Proposition~3 in \cite{Lancien2006}) yields $\Phi(L_2(Z)) \leq \omega^2\cdot\omega = \omega^3$, which concludes the proof.
    \end{proof}

We now proceed to show that this upper bound is optimal. We will use the same examples as those given for the weak-fragmentability index, based on the  use of the family $(D_n^\omega)_{n\in \N}$. The key general statement is the following.

\begin{prop}\label{p:SufficientForDentabilityOmegaThree} 
Assume that the family $(D_n^\omega)_{n\in N}$ embeds with the same distortion into a metric space $(M,d)$. Then 
$D(\F(M))\geq \omega^3$. 
\end{prop}

In the proof we will need the following lemma.

\begin{lemma}\label{l:DentabilityLowerBound}
Let us assume that a metric space $(M_n,d)$, a map $f:(D^\omega_n,d_n) \to (M,d)$ and a constant $A\leq 1$ satisfy
\[
Ad_n(x,y)\leq d(f(x),f(y))\leq d_n(x,y)
\]
for all $x,y \in D_n^\omega$. Then 
$$
f_{t_n,b_n}\in d_{A}^{\omega\cdot n}(B_{\F(M_n)}),
$$
where 
$$f_{t_n,b_n}=\frac{\delta(f(t_n))-\delta(f(b_n))}{d_n(t_n,b_n)}=m_{f(t_n),f(b_n)}\frac{d(f(t_n),f(b_n))}{d_n(t_n,b_n)}.$$ 
In particular, $\mol{t_n, b_n} \in d_1^{\omega\cdot n}(B_{\F(D_n^\omega)})$.
\end{lemma}

\begin{proof} 

We prove the statement by induction on $n$. 
It is clearly true for $D_0^\omega$. 
So let us assume that, for $n\in \N$,  it is true for $D_{n-1}^\omega$ and let us prove it for $D_n^\omega$. 
We use the notation from Section \ref{sec:diamonds}: $t_n,b_n$ are the top and bottom points of $D_n^\omega$ and $x_n^0,x_n^1,\ldots,x_n^i,\ldots$ the ``middle points'' of $D_n^\omega$. 
Note that $d_n(t_n,x_n^i)=d_n(b_n,x_n^i)=2^{n-1}$ and $d_n(t_n,b_n)=d_n(x_n^i,x_n^j)=2^n$, for $i \neq j$.
We denote 
$$u_i=f_{t_n,x_n^i}=\frac{1}{2^{n-1}}(\delta(f(t_n))-\delta(f(x_n^i))).$$
Since $t_n$ and $x_n^i$ are top and bottom points of a bi-Lipschitz
copy of $D^\omega_{n-1}$ in $M_n$, we have by induction hypothesis that $u_i\in d_{A}^{\omega\cdot (n-1)}(B_{\F(M_n)})$. 
We will show that for all $k\ge 0$ and all $i_1<\cdots<i_{2^k}$
\begin{equation}\label{eq:claim}
 \frac{1}{2^k}\sum_{j=1}^{2^k} u_{i_j} \in  d_{A}^{\omega\cdot (n-1)+k}(B_{\F(M_n)}).
\end{equation}
The proof goes by a second induction (on $k$). 
By the assumption of the first induction (on $n$), it is true for $k=0$. So assume it is true for $k\ge 0$ and let us prove it for $k+1$. So let  $i_1<\cdots<i_{2^{k+1}}$. 
Denote 
$$v=\frac{1}{2^k}\sum_{j=1}^{2^k} u_{i_j}\ \ \text{and}\ \ w=\frac{1}{2^k}\sum_{j=2^k+1}^{2^{k+1}} u_{i_j}.$$
By induction hypothesis $v$ and $w$ belong to $d_{A}^{\omega\cdot (n-1)+k}(B_{\F(M_n)})$ and so does $\frac12(v+w)$ by convexity of $d_{A}^{\omega\cdot (n-1)+k}(B_{\F(M_n)})$. 
Note that 
$$w-v=\frac1{2^{n-1}}\frac{1}{2^k}\Big(\sum_{j=1}^{2^k}\delta(f(x_n^{i_j})) - \sum_{j=2^k+1}^{2^{k+1}}\delta(f(x_n^{i_j}))\Big).$$
We claim that $\|w-v\|\ge 2A$. 
This follows from the fact that there exist $\varphi \in \Lip_0(M_n)$ and a constant $p \in \Real$ such that $\|\varphi\|_L\leq 1$,  $\varphi(f(x_n^{i_j}))=p$ for $1 \le j \le 2^k$ and $\varphi(f(x_n^{i_j}))=p-2A2^{n-1}$ for $2^k+1 \le j \le 2^{k+1}$. 
Thus $\langle w-v,\varphi\rangle=2A$ and $\|w-v\|\ge 2A$. Finally, we apply Lemma \ref{lem:midpointtrick} to deduce that $\frac12(v+w) \in d_A^{\omega\cdot (n-1)+k+1}(B_{\F(M_n)})$. 
This finishes the inductive proof of \eqref{eq:claim}. 

Therefore we have in particular that for all $k\in \N$:
$$\frac{1}{2^k}\sum_{j=1}^{2^k} u_{j}=\frac{1}{2^{n-1}}\Big(\delta(f(t_n)) - \frac{1}{2^k}\sum_{j=1}^{2^k} \delta(f(x_n^j))\Big) \in  d_A^{\omega\cdot (n-1)+k}(B_{\F(M_n)}).$$
By a similar argument we get:
$$\frac{1}{2^{n-1}}\Big(\frac{1}{2^k}\sum_{j=1}^{2^k} \delta(f(x_n^j)) -\delta(f(b_n))\Big) \in  d_A^{\omega\cdot (n-1)+k}(B_{\F(M_n)}).$$
Note that $f_{t_n,b_n}=\frac{1}{2^n}(\delta(f(t_n))-\delta(f(b_n)))$ is the midpoint of the two above elements of $d_A^{\omega\cdot(n-1)+k}(B_{\F(M_n)})$, which is a convex set. 
It follows that for all $k \in \N$, $f_{t_n,b_n}\in d_A^{\omega\cdot(n-1)+k}(B_{\F(M_n)})$. 
Thus $f_{t_n,b_n}\in d_A^{\omega\cdot n}(B_{\F(M_n)})$, which finishes the proof of the proposition. 
\end{proof}


\begin{proof}[Proof of Proposition~\ref{p:SufficientForDentabilityOmegaThree}] It follows from Lemma~\ref{l:DentabilityLowerBound} that there exists $A\in (0,1]$ so that for all $n\in \N$, there exists a scaling factor $s_n>0$ such that $d_A^{\omega\cdot n}(B_{\F(M,s_nd)})\neq \emptyset$. But $\F(M,s_nd)$ is linearly isometric to $\F(M,d)$, so $d_A^{\omega\cdot n}(B_{\F(M,d)})\neq \emptyset$. Since the derived sets are convex and symmetric, we get that for all $n\in \N$, $0 \in d_A^{\omega\cdot n}(B_{\F(M,d)})$ and therefore $0 \in d_A^{\omega^2}(B_{\F(M,d)})$.  Thus $D(\F(M,d))>\omega^2$. Since the dentability index of a Banach space is always of the form $\omega^\alpha$ (see \cite{Lancien2006} and references therein) we deduce that  $D(\F(M,d))\ge \omega^3$. 
\end{proof}


Exactly as in the previous section we deduce the following.

\begin{theorem}\label{th:D(F(grid))} Whenever $M$ is one of the following uniformly discrete metric spaces:
\begin{enumerate}
    \item the $\ell_1$-sum of the $D_n^\omega$'s,
    \item a net in $c_0$ (for example $\Z_{c_0}$),
    \item a net in $L_1$ or $\Z_{Haar}$,
\end{enumerate}
then the dentability index of $\F(M)$ is equal to $\omega^3$.  
\end{theorem}

\begin {remark}\label{rem:p-1-uProper}
Let us briefly make comments that are analogous to Remark \ref{rem:Phi}. We know from \cite{Basset} that the set of possible values for $D(\F(M))$, for $M$ complete separable purely 1-unrectifiable is uncountable. It is even uncountable if we restrict ourselves to countable discrete and complete metric spaces (see next section). Again computing $D(\F(M))$ is an interesting problem. The first intriguing case is the class of metric spaces $M$ such that $D(\F(M))=\omega^2$. We know that it contains the class of proper p-1-u metric spaces. Indeed, we have already explained in Remark~\ref{rem:Phi} that if $M$ is proper p-1-u, then $\F(M)$ admits a dual weak$^*$ asymptotically uniformly convex norm. 
Then it follows from results in \cite{HajekLancien} that the weak$^*$-dentablity index and therefore the dentability index of $\F(M)$ cannot exceed $\omega^2$.

We do not know whether the conditions $\phi(\F(M))=\omega$ and $D(\F(M))=\omega^2$ are equivalent. 
\end{remark}

\section{Countable complete discrete metric spaces}

The main purpose of this section is to prove the following.

\begin{theorem}\label{thm:mainexample} Let $\xi$ be a countable ordinal. Then there exists a complete discrete metric space $D_\xi$ such $D(\F(D_\xi))> \xi$.
\end{theorem}

\subsection{Construction of the spaces}

It will be convenient to set up the following notation.
\begin{definition}
     Let $M$ be a metric space that contains points $t,b$.
    Let $k\in \Natural$. We will call the \emph{chain of $2^k$ copies of $M$} the metric space $N$ obtained as follows. 
    Let $t_i,b_i \in M_i$ for $i=1,\ldots, 2^k$ be disjoint copies of $M$ together with their corresponding distinguished points.
    Then $N=\bigcup_{i=1}^{2^k} M_k$ in which we identify $t_i=b_{i+1}$ for every $i=1,\ldots,2^k-1$. 
    We define $d_N(x,y)=d_{M_i}(x,y)$ if $x,y \in M_i$ and $$d_N(x,y)=d(x,t_i)+(j-i-1)d(t,b)+d(b_j,y)$$ if $x\in M_i$, $y\in M_j$ and $i<j$.
\end{definition}

For any given $\xi \in [1,\omega_1)$ we shall construct a family $(M_\alpha)_{\alpha \le \xi}$ of complete discrete metric spaces such that for all $\alpha \le \xi$, $d_{1/2}^\alpha(B_{\F(M_\alpha)})\neq \emptyset$. Then we will set $D_\xi=M_\xi$. Let us point out that an important feature of our construction is that the family $(M_\alpha)_{\alpha \le \xi}$ depends on $\xi$. Nevertheless, for the sake of readability, our notation does not reflect this dependence. 

\medskip So let us fix $\xi \in (0,\omega_1)$. Let $\Gamma$ be the set of limit ordinals in $(0,\xi]$. 
If $\Gamma \neq \emptyset$, pick $(\eps_\gamma)_{\gamma \in \Gamma}$ in $\{2^{1-k},\ k\in \N\}$ such that $\prod_{\gamma \in \Gamma}(1-\eps_\gamma) \ge \frac12$ and for every $\alpha \le \xi$ denote $$\mu_\alpha=\prod_{\gamma \le \alpha, \gamma \in \Gamma}(1-\eps_\gamma).$$
If $\Gamma=\emptyset$, that is $\xi=n \in \N$, we set $\mu_n=1$.

We will define $(M_\alpha,d_\alpha)_{\alpha \le \xi}$ inductively. For each $\alpha \le \xi$, we will distinguish two special points: $t_\alpha$ (top) and $b_\alpha$  (bottom), so that $1=d_\alpha(t_\alpha,b_\alpha)=\diam(M_\alpha)$.  The origin of the space $M_\alpha$ will always be $b_\alpha$. 
We define $M_0=\{0,1\}$ equipped with the distance induced by $\R$ and let $b_1=0$ and $t_1=1$. 
Let $\alpha \le \xi$ and assume that $(M_\beta,d_\beta)$, together with $t_\beta$ and $b_\beta$ have been constructed for all $\beta<\alpha$.

Assume first that $\alpha=\beta+1\le \xi$ is a successor ordinal. 
Then $M_\alpha$ is defined as the chain of $2$ copies of $(M_\beta,\frac12 d_\beta)$. We call these copies $M_{\beta,1}$ and $M_{\beta,2}$ and we set $b_\alpha=b_{\beta,1}$ and $t_\alpha=t_{\beta,2}$.

Assume now that $\alpha$ is a limit ordinal, that is $\alpha \in \Gamma$. Then there exists $k \in \N$ such that $\eps_\alpha=2^{1-k}$. Let $(\beta_n)_{n=1}^\infty$ be a sequence in $(0,\alpha)$ such that $\sup_n \beta_n=\alpha$. We may as well assume that $\beta_n$ is of the form $\gamma_n+k_n$ for some $\gamma_n<\alpha$ and $k_n \in \N$ such that $k_n>2k$ and $\sup_n k_n=\omega$. 
Then note that $M_{\beta_n}$ is the chain of $2^{k_n}$ copies of $(M_{\gamma_n},2^{-k_n}d_{\gamma_n})$. 
Each one of the copies $M_{\gamma_n,i}$ is equipped with top and bottom points $t_{\gamma_n,i}$ and $b_{\gamma_n,i}$.
Next we consider the following metric subspace of $M_{\beta_n}$: 
$$P_{\beta_n}=\{x \in M_{\beta_n},\ \min\{d_{\beta_n}(x,t_{\beta_n}),d_{\beta_n}(x,b_{\beta_n})\}\ge 2^{-k}\}\cup \{t_{\beta_n},b_{\beta_n}\}.$$
The space $P_{\beta_n}$ is equipped with the metric induced by $d_{\beta_n}$, still denoted $d_{\beta_n}$.  In other words: we uniformly isolate (by $2^{-k}$) the points $t_{\beta_n}$ and $b_{\beta_n}$ in $M_{\beta_n}$. We can also write 
$$P_{\beta_n}=\{t_{\beta_n},b_{\beta_n}\} \cup \bigcup_{i=2^{-k}2^{k_n}+1}^{(1-2^{-k})2^{k_n}}M_{\gamma_n,i}.$$
Finally, $M_\alpha$ is obtained by gluing together the spaces $P_{\beta_n}$ by their top and bottom points. So we set $M_\alpha=\bigcup_{n\in \N} P_{\beta_n}$  with the identifications  $t_\alpha=t_{\beta_n}$ and $b_{\alpha}=b_{\beta_n}$ for all $n\in \N$. For $x,y \in P_{\beta_n}$, we set $d_\alpha(x,y)=d_{\beta_n}(x,y)$ and for $x \in P_{\beta_n}$ and $y \in P_{\beta_m}$ with $n\neq m$, we set 
$$d_\alpha(x,y)=\min\big\{d_{\beta_n}(x,b_{\beta_n})+d_{\beta_m}(b_{\beta_m},y), d_{\beta_n}(x,t_{\beta_n})+d_{\beta_m}(t_{\beta_m},y)\big\}.$$

\subsection{\texorpdfstring{Properties of the family $(M_\alpha)_{\alpha \le \xi}$}{Properties of the family}}\ \\
Throughout this subsection $\xi \in (0,\omega_1)$ is fixed. 

\begin{prop}
For every $\alpha \in (0,\xi]$, $M_\alpha$ is countable discrete and complete.   
\end{prop}

\begin{proof} Each of these properties will be proved by transfinite induction. The countability is obvious, so let us prove that for  $\alpha \in [0,\xi]$, $M_\alpha$ is discrete. 
Clearly $M_0$ is discrete. 
Let $\alpha \in  (0,\xi]$ and assume that for all $\beta<\alpha$, the space $M_\beta$ is discrete. Assume first that $\alpha=\beta+1$. 
Since $M_\beta$ is discrete, it is  clear that the gluing of $M_{\beta,1}$ and $ M_{\beta,2}$ as defined in the construction is also discrete. Assume now that $\alpha$ is a limit ordinal and let us use the notation introduced in the construction of $M_\alpha$. Let $x \in M_\alpha$. If $x$ is either $t_\alpha$ or $b_\alpha$, then it is $2^{-k}$ separated from the rest of $M_\alpha$ and therefore isolated. Otherwise, there exists $n \in \N$ such that $x\in P_{\beta_n} \setminus \{b_{\beta_n},t_{\beta_n}\}$. Since $M_{\beta_n}$ is discrete, $x$ is isolated in $M_{\beta_n}$ and therefore in $P_{\beta_n} \setminus \{b_{\beta_n},t_{\beta_n}\}$. It is also $2^{-k}$ separated from the other points of $M_\alpha$ and therefore isolated in $M_\alpha$.

Let us now justify the completeness of $M_\alpha$. Obviously $M_0$ is complete. 
Let $\alpha \in  (0,\xi]$ and assume that for all $\beta<\alpha$, the space $M_\beta$ is complete. Assume first that $\alpha=\beta+1$ and let $(x_l)_l \subset M_\alpha$ be a Cauchy sequence for $d_\alpha$. Then, there exists $i\in \{1,2\}$ and a subsequence $(y_l)_l$ of $(x_l)_l$ included in $M_{\beta,i}$. Then $(y_l)_l$ is Cauchy in $M_{\beta,i}$ for $d_{\beta,i}=\frac12d_\beta$ and, since $(M_\beta,d_\beta)$ is complete, $(y_l)_l$ is converging in $(M_{\beta,i},d_{\beta,i})$. Since $(x_l)_l$ is Cauchy and admits a converging subsequence in $(M_\alpha,d_\alpha)$, it is converging in $(M_\alpha,d_\alpha)$. Assume now that $\alpha$ is a limit ordinal and let us use again the notation introduced in the construction of $M_\alpha$. Let $(x_l)_l \subset M_\alpha$ be a Cauchy sequence for $d_\alpha$. There exists $l_0\in \N$ such that for all $l,m \ge l_0$,  $d_\alpha(x_l,x_m)<2^{-(k+1)}$. Thus either $x_l=t_\alpha$ for all $l\ge l_0$, or $x_l=b_\alpha$ for all $l\ge l_0$, or there exists $n\in \N$ such that for all $l\ge l_0$, $x_l \in \{x \in M_{\beta_n},\ \min\{d_{\beta_n}(x,t_{\beta_n}),d_{\beta_n}(x,b_{\beta_n})\}\ge 2^{-k}\}$. Since, $\{x \in M_{\beta_n},\ \min\{d_{\beta_n}(x,t_{\beta_n}),d_{\beta_n}(x,_{\beta_n})\}\ge 2^{-k}\}$ is a closed subset of the complete metric space $(M_{\beta_n},d_{\beta_n})$, it follows that $(x_l)_l$ is converging in $(M_\alpha,d_\alpha)$. 
    
\end{proof} 

\begin{prop}\label{p:KeyDentabilityLowerBound} 
For every $\alpha \leq \xi$ we have  $m_{t_\alpha,b_\alpha} \in d_{\mu_\alpha}^\alpha(B_{X_\alpha})$, where $X_\alpha=\F(M_\alpha)$.  
\end{prop}

In the proof we will use repeatedly the following lemma.
\begin{lem}\label{l:iterated}
    Let $M$ be a metric space that contains points $t,b$ with the property that $m_{t,b}\in d_\varepsilon^\alpha(B_{\Free(M)})$ for some ordinal $\alpha$ and some $\varepsilon \in (0,1]$.
    Let $l\in \Natural \cup \set{0}$, $s>0$ and let $N=N(l)$ be the chain of $2^l$ copies of $(M,s\cdot d)$. 
    Denote $t_N=t_{2^l}$ and $b_N=b_1$.
    Let $W$ be a metric space and $f:N \to W$ an isometric embedding.
    Then $m_{f(t_N),f(b_N)}\in d_\varepsilon^{\alpha+l}(B_{\F(W)})$.
\end{lem}
\begin{proof}
    Since the spaces $\Free(M,s\cdot d)$ are isometrically isomorphic for different choices of $s$, we might as well assume that $s=1$.

    The proof goes by an easy inductive argument in $l$ using Lemma~\ref{lem:midpointtrick}.
    For $l=0$ there is almost nothing to prove.
    Observe that both the assumption and the conclusion are independent of the actual base point of $M$, resp. $W$. We may therefore assume that $f(0)=0$. Then $\Free(N)$ is linearly isometric to $\Free(f(N))$ which is canonically a closed linear subspace of $\Free(W)$. The conclusion follows by standard facts (in particular the fact that $d_\varepsilon(B_X) \cap Y=d_\varepsilon(B_Y)$ when $Y\subseteq X$).
    Assume now that we have proved the claim for some $l$. 
    Notice that $N=N(l+1)$ is a chain of 2 copies of $N(l)$, say $P_1, P_2$.
    We denote $t^P_i,b^P_i$ the corresponding top and bottom points.
    Let $f:N \to W$ be an isometric embedding.
    Then for $i=1,2$ we have $m_{f(t^P_i),f(b^P_i)} \in d_\varepsilon^{\alpha+l}(B_{\Free(W)})$ by the inductive hypothesis.
    Since \[m_{f(t^P_2),f(b^P_1)}=\frac12\left(m_{f(t^P_2),f(b^P_2)}+m_{f(t^P_1),f(b^P_1)}\right),\]
    Lemma~\ref{lem:midpointtrick} implies that every slice containing $m_{f(t^P_2),f(b^P_1)}$ will contain one of these molecules.
    Since $\big\|m_{f(t^P_1),f(b^P_1)}-m_{f(t^P_2),f(b^P_2)}\big\|=2$, this implies that \[m_{f(t_N),f(b_N)}=m_{f(t^P_2),f(b^P_1)} \in d_1(d_\varepsilon^{\alpha+l}(B_{\Free(W)}))\subset d_\varepsilon^{\alpha+l+1}(B_{\Free{(W)}}).\]

\end{proof}

\begin{proof}[Proof of Proposition~\ref{p:KeyDentabilityLowerBound}] 
This will again be a transfinite induction on $\alpha$.
The claim is trivially true for $M_0$.
Let $\alpha \in  (0,\xi]$ and assume that the statement has been proved for all $\beta<\alpha$. Assume first that $\alpha=\beta+1$. 
Note that $\mu_\alpha=\mu_\beta$. 
Applying Lemma~\ref{l:iterated} for $M=M_\beta$, $\varepsilon=\mu_\alpha \in (0,1]$, $l=1$, $s=\frac12$, $W=M_\alpha$ and $f=Id_{M_\alpha}$ we get
that $m_{t_\alpha,b_\alpha} \in d_{\mu_\alpha}^{\beta+1}(B_{X_\alpha})$.
This finishes the successor case.\\
Assume now that $\alpha$ is a limit ordinal and let us use one more time the notation from the construction of $M_\alpha$. 
Let us introduce two new particular points in $P_{\beta_n}$. 
We denote $u_{\beta_n}=t_{\gamma_n,(1-2^{-k})2^{k_n}}$, which is the closest point to $t_{\beta_n}$ in $P_{\beta_n}\setminus \{t_{\beta_n},b_{\beta_n}\}$ and $v_{\beta_n}=b_{\gamma_n,2^{-k}2^{k_n}+1}$, which is the closest point to $b_{\beta_n}$ in $P_{\beta_n}\setminus \{t_{\beta_n},b_{\beta_n}\}$. Note that $d_\alpha(t_{\beta_n},u_{\beta_n})=d_\alpha(b_{\beta_n},v_{\beta_n})=2^{-k}$ and $d_\alpha(u_{\beta_n},v_{\beta_n})=1-2^{1-k}$. 
Also note that $P_{\beta_n}\setminus \set{t_{\beta_n},b_{\beta_n}}$ is a chain of $2^{k_n-k}$ copies of $(M_{\gamma_n},s\cdot \gamma_n)$ for $s=2^{-k_n}$.
By induction hypothesis, $m_{t_{\gamma_n},b_{\gamma_n}} \in d_{\mu_{\gamma_n}}^{\gamma_n}(B_{X_{\gamma_n}})$.

A direct application of Lemma~\ref{l:iterated} reveals that
$m_{u_{\beta_n},v_{\beta_n}} \in d_{\mu_{\gamma_n}}^{\gamma_n+k_n-k}(B_{X_\alpha})$.
Recall now that $d_\alpha(t_{\beta_n},u_{\beta_n})=d_\alpha(b_{\beta_n},v_{\beta_n})=2^{-k}$ and $d_\alpha(u_{\beta_n},v_{\beta_n})=1-2^{1-k}$. 
It follows that 
\[m_{t_{\beta_n},b_{\beta_n}}=2^{-k}m_{t_{\beta_n},u_{\beta_n}}+(1-2^{1-k})m_{u_{\beta_n},v_{\beta_n}}+2^{-k}m_{v_{\beta_n},b_{\beta_n}}.\] 
Thus, recalling that $\eps_\alpha=2^{1-k}$ in the notation from the construction of $M_\alpha$,  we get:
$$m_{t_{\beta_n},b_{\beta_n}}\in (1-2^{1-k})d_{\mu_{\gamma_n}}^{\gamma_n+k_n-k}(B_{X_\alpha})+ 2^{1-k}B_{X_\alpha}=(1-\eps_\alpha)d_{\mu_{\gamma_n}}^{\gamma_n+k_n-k}(B_{X_\alpha})+ \eps_\alpha B_{X_\alpha}.$$
We can now apply Proposition \ref{prop:conv_comb} and the fact that $(1-\eps_\alpha)\mu_{\gamma_n} \geq  \mu_\alpha$ to deduce that for all $n\in \N$, 
$$m_{t_{\alpha},b_{\alpha}}\in d_{(1-\eps_\alpha)\mu_{\gamma_n}}^{\gamma_n+k_n-k}(B_{X_\alpha}) \subset d_{\mu_\alpha}^{\gamma_n+k_n-k}(B_{X_\alpha}).$$
Since $k_n$ tends to $\omega$, we have that 
$$\bigcap_{n=1}^\infty d_{\mu_\alpha}^{\gamma_n+k_n-k}(B_{X_\alpha})= d_{\mu_\alpha}^{\alpha}(B_{X_\alpha}).$$
Thus $m_{t_{\alpha},b_{\alpha}}\in d_{\mu_\alpha}^{\alpha}(B_{X_\alpha})$. 

\end{proof}

\begin{proof}[Proof of Theorem~\ref{thm:mainexample}]
    Since all $\mu_\alpha\geq \frac12$, Proposition~\ref{p:KeyDentabilityLowerBound} implies that $d_{\frac12}^\xi(B_{\Free(M_\xi)})\neq \emptyset$. Thus $D(\Free(M_\xi))>\xi$.
\end{proof}

\subsection{Applications} Since any family of countable ordinals that is not bounded by a countable ordinal must be uncountable, we deduce immediately the following.

\begin{corollary} There exists an uncountable family of complete, discrete and countable metric spaces whose free spaces are mutually non isomorphic.
\end{corollary}

We continue  with a result on bi-Lipschitz universality for the class of  complete, discrete and countable metric spaces. 
\begin{corollary}\label{c:NoUniversal} If a complete separable metric space $M$ is universal for the class of complete countable discrete metric spaces and bi-Lipschitz embeddings, then $M$ is not purely 1-unrectifiable.
\end{corollary}

\begin{proof}
Let $M$ be a complete separable purely 1-unrectifiable metric space. Then, by Theorem C in \cite{AGPP}, $\F(M)$ has RNP and by Proposition  \ref{prop:RNP_D(X)}, there exists $\alpha<\omega_1$ such that $D(\F(M))=\alpha$. But, by Theorem  \ref{thm:mainexample}, $D(\F(D_\xi))> D(\F(M))$, for all $\xi>\alpha$. Thus, for $\xi >\alpha$, $\F(D_\xi)$ does not linearly embed into $\F(M)$ and in particular, $D_\xi$ is a complete discrete separable metric space which does not bi-Lipschitz embed into $M$. 
\end{proof}

Another consequence of this work is that the free space of a complete countable discrete metric space can be very different from the free space of a uniformly discrete metric space or of a proper metric space. In particular:

\begin{corollary}\label{c:CCD-cpt} There exists a complete countable discrete metric space $D$ such that if $\F(D)$ linearly embeds into $\Free(M)$ for some metric space $M$, then $M$ is neither uniformly discrete, nor proper and p-1-u. 
Also if $\Free(D)$ is isomorphic to $\Free(M)$, then $M$ is not proper.
\end{corollary}

\begin{proof} Let $\omega^3<\xi<\omega_1$. It follows from Theorem \ref{thm:mainexample} and Proposition~\ref{c:DentabilityCUD}  that $\F(D_\xi)$ does not linearly embed into the free space of a uniformly discrete metric space.

Assume that $M$ is proper p-1-u. Then, as we mentioned in Remark \ref{rem:p-1-uProper}, $D(\F(M))\le \omega^2$. Thus, if $\omega^2<\xi<\omega_1$, $\F(D_\xi)$ does not linearly embed into $\F(M)$.

Assume now that $M$ is proper and $\F(M)$ is isomorphic to $\F(D_\xi)$ for some $\xi<\omega_1$. Then $\F(M)$ has RNP and, by \cite{AGPP}, $M$ is p-1-u. The previous argument showed that this is impossible for $\omega^2<\xi<\omega_1$.

\end{proof}

\section*{Acknowledgments}
We wish to thank Eva Perneck\'a and Ram\'on Aliaga for suggesting the question that led to Corollary \ref{c:CCD-cpt} and for many insightful discussions while the second and third named authors were visiting the Czech Technical University in Fall 2024.  

The three authors were partially supported by the French ANR projects No. ANR-20-CE40-0006 and ANR-24-CE40-0892-01.

\end{document}